\documentclass[11pt,reqno]{amsproc}

\usepackage[margin=1.0in]{geometry}
\usepackage[english]{babel}

\usepackage[utf8]{inputenc}
\usepackage{amsmath,amssymb,amsthm, comment, mathtools}
\usepackage{amsthm}
\usepackage{hyperref}
\usepackage{mathrsfs}
\usepackage{slashed,xcolor}
\usepackage{units}

\usepackage{etoolbox}
\expandafter\patchcmd\csname\string\proof\endcsname
  {\normalparindent}{0pt }{}{}

\newtheorem*{thm*}{Theorem}
\newcommand{\pa}{\partial}
\newcommand{\la}{\label}
\newcommand{\fr}{\frac}

\newcommand{\R}{{\mathbb R}}
\newtheorem{thm}{Theorem}
\newtheorem{prop}{Proposition}

\newtheorem{lemma}{Lemma}
\newtheorem{cor}{Corollary}

\theoremstyle{definition}
\newtheorem{rem}{Remark}

\newcommand{\ov}{\overline}
\renewcommand{\div}{{\mbox{div}\,}}

\newcommand{\rmd}{{\rm d}}

\def\curl{{\rm curl}}
\def\div{\text{div}}

\newcommand{\be}{\begin{equation}}
\newcommand{\ee}{\end{equation}}

\newcommand{\ba}{\begin{array}{l}}
\newcommand{\ea}{\end{array}}

\title[]{Flexibility and rigidity of free boundary MHD equilibria}

\author{Peter Constantin}
\address{Department of Mathematics, Princeton University, Princeton, NJ 08544}
\email{const@math.princeton.edu}
\author{Theodore D. Drivas}
\address{Department of Mathematics, Stony Brook University,
Stony Brook, NY, 11794}
\email{tdrivas@math.stonybrook.edu}
\author{Daniel Ginsberg}
\address{Program in Applied and Computational Mathematics, Princeton University, Princeton, NJ 08544}
\email{ dg42@princeton.edu}

\date{today}

\date{today}
\begin{document}

\begin{abstract}
We study stationary free boundary configurations of an ideal incompressible magnetohydrodynamic fluid possessing nested flux surfaces.  In 2D simply connected domains, we prove that if the magnetic field and velocity field are never commensurate, the only possible domain for any such equilibria is a disk, and the velocity and magnetic field are circular.  We give examples of non-symmetric equilibria occupying a domain of any shape by imposing an external magnetic field generated by a singular current sheet  charge distribution (external coils).   Some results carry over to 3D axisymmetric solutions. These results highlight the importance of external magnetic fields for the existence of asymmetric equilibria.
\end{abstract}

\maketitle

We consider an ideal incompressible magnetohydrodynamic fluid occupying a domain $\Omega \subset \mathbb{R}^d$, $d=2,3$ with free boundary $\partial \Omega$ as a model for plasma equilibria.
We assume that the region $\Omega$ is surrounded by vacuum containing a
 current $J_{\mathsf{ext}}$ (scalar valued in dimension two and vector valued in dimension three) supported away from $\Omega$.  The external field $ B_{\mathsf{ext}}$ satisfies
\begin{alignat}{2}\label{mhd11}
\ \qquad \qquad {\rm curl} \ B_{\mathsf{ext}}  &= J_{\mathsf{ext}}\qquad\qquad\   &&\text{in} \ \ \mathbb{R}^d \setminus \Omega,\\
  \nabla \cdot B_{\mathsf{ext}}  &= 0&& \text{in} \ \ \mathbb{R}^d \setminus \Omega,\label{mhd2}
\end{alignat}
 where ${\rm curl} = \nabla^\perp \cdot$ with $\nabla^\perp :=(-\partial_2,\partial_1)$ in two dimensions.
Plasma equilibria are then governed by
\begin{alignat}{2}\label{mhd1}
u \cdot \nabla u  - B\cdot \nabla B&= -\nabla p && \text{in} \ \ \Omega,\\\label{mhd2}
u \cdot \nabla B - B\cdot \nabla u &= 0&& \text{in} \ \ \Omega, \\ \label{mhd3}
\nabla \cdot B &= 0 &&\text{in} \ \ \Omega, \\ \label{mhd4}
\nabla \cdot u &= 0 &&  \text{in} \ \ \Omega, \\\label{mhd5}
p&=\tfrac{1}{2}|B_{\mathsf{ext}}  |^2 \qquad\quad && \text{on} \ \ \partial \Omega, \\\label{mhd6}
B \cdot \hat{n}&=B_{\mathsf{ext}}  \cdot \hat{n}&&\text{on} \ \ \partial \Omega,\\\label{mhd7}
 u \cdot \hat{n}&=0 && \text{on} \ \ \partial \Omega,
\end{alignat}
where  $\hat{n}$ is the unit outward normal to the boundary.
The conditions \eqref{mhd5}--\eqref{mhd7} arise from demanding that \eqref{mhd1}--\eqref{mhd4} hold globally in the weak sense. Specifically, condition \eqref{mhd5} comes from noting that away from the support of $J_{\mathsf{ext}}$, the external magnetic field satisfies
\be
  -B_{\mathsf{ext}} \cdot \nabla B_{\mathsf{ext}} = -\frac{1}{2}\nabla  |B_{\mathsf{ext}}|^2 \qquad \   \text{in} \ \ \mathbb{R}^d \setminus {\rm supp}(J_{\mathsf{ext}}).
\ee
Therefore, the pressure $p_{\mathsf{ext}}$ in the vacuum is simply $\tfrac{1}{2}|B_{\mathsf{ext}}|^2$ and  \eqref{mhd5} comes from demanding that the jump $[p]:=p-p_{\mathsf{ext}}$ in the pressure across the free boundary is zero.
The condition \eqref{mhd6} forces the penetrative component of the magnetic field to be continuous across the surface.
We will consider here only the case where the external field does not penetrate the plasma, i.e.
\be
B_{\mathsf{ext}}  \cdot \hat{n}=0\quad\quad   \text{on} \ \ \partial \Omega.
\ee
The reason we restrict to this case is that, according to \eqref{mhd6}, $B \cdot \hat{n}=0$ on the boundary so that these free boundary equilibria are, in fact, a subclass of fixed boundary equilibria on a given domain.

In the formulation \eqref{mhd11}--\eqref{mhd7}, we have allowed for an external magnetic field sourced by an imposed current distribution $J_{\mathsf{ext}}$.  The motivation for including such effects comes from the
plasma confinement fusion program.
Free boundary plasma equilibria which are confined to finite volumes and surrounded by vacuum are the principal objects in this study. Their creation and maintenance is a huge scientific and engineering challenge, with the Tokamak, Spheromak and Stellarator being the most popular designs to accomplish this. In all situations, but in particular the Stellarator, it is necessary to keep the plasma in a contoured shape which may not naturally be stable or even stationary.  As such, it is necessary to impose external currents driven through coils surrounding the device to hold the configuration steady.

In this work, we address the question of which configurations are naturally stationary  without external coils  (Theorem \ref{rigthm2d}) and how to prescribe or a coil geometry (Theorem \ref{nonsymmthmext})  in order to hold contortions of these steady.  All of this is done primarily in two dimensions and included possible non-trivial flow velocity. Some of the results can be extended to certain three-dimensional geometries with symmetries. For $z$--independent axisymmetric solutions occupying infinite cylindrical domains, we prove a similar rigidity result as in $2d$ (Theorem \ref{3dthm1}). For $\varphi$--independent axisymmetric solutions possibly occupying toroidal domains, we prove the solution and domain must have up-down symmetry  (Theorem \ref{3dthm2}).\footnote{We remark that there are well known ``virial"  theorems (see Appendix \ref{virial}) which also expose a certain form of rigidity. In particular, these results can be used to rule out any nontrivial solutions without external forcing if one makes the additional assumption that the pressure is non-negative.  In particular, for a compressible gas (magnetized or not), the pressure is given by an equation of state involving the density and is thus positive.   As such, this theorem implies non-existence of free boundary equilibria.  On the other hand, for the incompressible medium studied here, the assumption that the pressure is positive is not natural and there indeed exist many free boundary steady states.
}

Finally we mention a conjecture of Grad
\cite{GR58,G67,G85} which asserts that in three dimensions, the only stationary equilibria, fibered by flux surfaces and without flow velocity, must inherit Euclidean symmetries (e.g. axisymmetry in the context of a toroidal plasma body).  We prove the analogous statement is, in fact, true in two-dimensions for free boundary solutions with a single magnetic nullpoint if one does not impose external currents (Theorem \ref{rigthm2d}). However, fixed boundary configurations (which can be made free boundary with the use of appropriate coils according to Theorem \ref{nonsymmthmext}) are flexible, see \cite{CDG21a}.  If similar results hold in three-dimensions, it would provide theoretical justification for the Stellarator program.

\section{Two dimensions}
In this section, we consider simply connected bounded domains $\Omega\subset \mathbb{R}^2$. In light of the incompressibility of $u$ and $B$, equations \eqref{mhd3} and \eqref{mhd4}, we can introduce a scalar streamfunction $\psi$ and magnetic potential $A$
\be
u = \nabla^\perp \psi , \qquad B = \nabla^\perp A.
\ee

Note that, in order to enforce the tangency of the velocity and magnetic field at the free boundary, \eqref{mhd6} and \eqref{mhd7},  we require that $\partial \Omega$ is an isosurface of $\psi$ and $A$ simultaneously.
An equivalent dynamical formulation consists of the vorticity equation and the advection-diffusion equation for the magnetic potential.
In this formulation, equations \eqref{mhd1}--\eqref{mhd4} are replaced by
\begin{align}\label{vort1}
 \{\psi ,\Delta \psi\}   -\{A ,\Delta A\}    &=0,\\ \label{vort2}
\{\psi,A\}&= 0,
\end{align}
where $\{a,b\}:= \nabla^\perp a \cdot \nabla b$ is the usual Poisson bracket, which is anti-symmetric in its arguments.
Our first result is a characterization of non-degenerate stationary solutions, with either free or fixed boundary. See Fig. \ref{bstruct} for a depiction of the type of magnetic field structures we consider.

    \begin{figure}[h!]\label{bstruct}
      \includegraphics[height=.3 \linewidth]{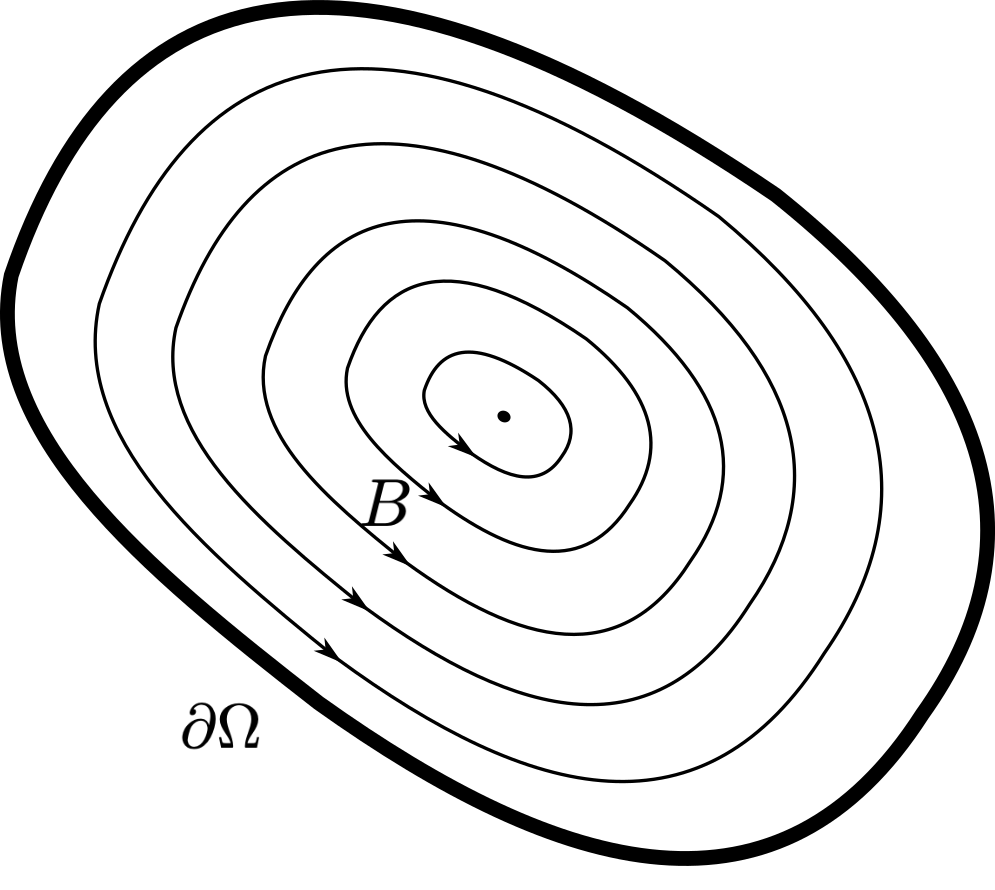}
    \end{figure}

\begin{prop}\label{structuretheorem} Let $\Omega$ be a  simply connected  bounded domain.
Let $u=\nabla^\perp \psi$ and $B=\nabla^\perp A$ be a $C^k$ with $k\geq 2$ steady solution of the system \eqref{mhd1}--\eqref{mhd7} without requiring the free boundary condition \eqref{mhd5}.  Suppose that $B$ is non-vanishing except at a single point with $k$-times differentiable travel time
\be\label{traveltime}
\mu(a) := \oint_{\{ A=a\}} \frac{\rmd\ell}{|\nabla A|} .
\ee
Then, there exists a $C^k$ function $G:\mathbb{R}\to \mathbb{R}$ such that
\be\label{Arel}
\psi = G(A),
\ee
and there exist a $C^{k-1}$ function  $F:\mathbb{R}\to \mathbb{R}$ such that $\psi$ solves
\begin{align}\label{psieqn}
\left(1-|G'(A)|^2\right)\Delta A - G'(A) G''(A) |\nabla A|^2 &= F(A) \qquad \quad \text{\rm in} \ \ \Omega,\\
A&=0 \qquad\quad \quad \ \   \text{\rm on} \ \ \partial \Omega.
\end{align}
Moreover, there exists   a $C^{k-1}$ function $H:\mathbb{R}\to \mathbb{R}$ such that the pressure satisfies
\be\label{pressure}
p+ \tfrac{1}{2}|u|^2 - \tfrac{1}{2}|B|^2 = H(A).
\ee
\end{prop}
\begin{rem}[Assumption on the travel-time]
 Away from the unique null point, the assumption on the travel time \eqref{traveltime} is automatic from the assumed regularity on the solution.  The content of the assumption lies in the order of degeneracy of the critical point of $A$ -- we exclude solutions which degenerate too rapidly.  In the radial case, this amounts to considering streamfunctions which vanish no faster that $r^2$. See Remark \ref{traveltimerem} for more discussion.  More generally, for  $C^\infty$ streamfunctions,  it is proved that any such satisfying our assumptions on the travel time must be Morse functions (Theorem 2.4 of \cite{V92}).  Moreover, for $C^k$ streamfunctions which are Morse, the travel time will be of class $C^{\lfloor \frac{k}{2} \rfloor -1}$ which follows from the finite-regularity extension of Vey's theorem appearing in Appendix D of \cite{KKPS20}.
   These rates of degeneracy, and therefore the regularity of $\mu$, are important in deducing the existence of $G$ and $F$ above in the claimed regularity classes.  This is the content of  Lemma \ref{structuretheorem1} of the Appendix.
\end{rem}
\begin{proof}[Proof of Proposition \ref{structuretheorem}]
The existence of a function $G$ such that the relation \eqref{Arel} holds follows from Lemma \ref{structuretheorem1}.
 In order to obtain \eqref{psieqn}, using  \eqref{Arel} we find that
\begin{align*}
\{\psi ,\Delta \psi\} &= \{G(A) ,\Delta G(A)\} \\
&=G'(A)\nabla^\perp A\cdot \nabla \Delta G(A)\\
&=\nabla^\perp A \cdot \nabla \left(G'(A)  \Delta G(A)\right).
\end{align*}
Thus from the vorticity equation \eqref{vort1} we find
\be
\nabla^\perp A \cdot \nabla \left( \Delta A- G'(A)  \Delta G(A)\right)=0.
\ee
Whence, appealing again to  Lemma \ref{structuretheorem1}, we deduce the existence of  an $F$ such that
\be
\Delta A- G'(A)  \Delta G(A) = F(A).
\ee
Expanding by chain rule gives the claimed equation  \eqref{psieqn}.

To find the pressure formula, we note that one rewrite the steady momentum equation as
\begin{align}
 u^\perp \omega  - B^\perp \eta&= -\nabla \left(p+ \tfrac{1}{2}|u|^2 - \tfrac{1}{2}|B|^2 \right)
\end{align}
where $\omega =\nabla^\perp \cdot u$ and  $\eta =\nabla^\perp \cdot B$.   Since $u^\perp= G'(A) B^\perp$ as a consequence of  \eqref{Arel},
dotting with $B$ we obtain
\be
B \cdot \nabla \left(p+ \tfrac{1}{2}|u|^2 - \tfrac{1}{2}|B|^2 \right)=0.
\ee
Thus by,  Lemma \ref{structuretheorem1}, equation \eqref{pressure} follows.
\end{proof}

Proposition \ref{structuretheorem} is a structure theorem for steady states with non-degenerate magnetic fields. An immediate corollary is
\begin{cor}\label{corrol} Let $\Omega$ be a  simply connected  bounded domain.
The fields $u=\nabla^\perp \psi$, $B=\nabla^\perp A$ and $\Omega$ form a $C^k$ steady free boundary solution of the system \eqref{mhd1}--\eqref{mhd7} with the property that $B$ is non-vanishing except at a single point with $k$-times differentiable travel time if there is a $C^k$ function $G$ and  some $C^{k-1}$ functions $F$ and $H$, such that $A$ solves
\begin{align}\label{Gpsieqn}
\left(1-|G'(A)|^2\right)\Delta A - G'(A) G''(A) |\nabla A|^2 &= F(A)  \qquad\qquad\qquad \quad \ \text{\rm in} \ \ \Omega,\\ \label{psieqn2}
A&=0 \qquad\quad\qquad\qquad \quad \ \  \   \text{\rm on} \ \ \partial \Omega,\\ \label{psieqn3}
 \tfrac{1}{2}(|G'(0)|^2-1) |\partial_nA|^2&=H(0) -  \tfrac{1}{2}|B_{\mathsf{ext}} |^2\qquad\  \ \   \text{\rm on} \ \ \partial \Omega.
\end{align}
\end{cor}
Corollary \ref{corrol} exposes the overdetermined nature  of free boundary MHD solutions with the stipulated structure (the field $A$ must satisfy a nonlinear elliptic equation with both Dirichlet and Neumann data).
We recall the seminal result of Serrin \cite{S71} about overdetermined elliptic problems.
\begin{thm}[Theorem 2 of \cite{S71}]
  \label{serrin}
  Let $\Omega$ be a $C^2$ simply connected domain and
$f,g:\mathbb{R}^2\to \mathbb{R}$ be given  $C^1$ functions.
    Suppose $\psi$ is a $C^2$ solution of
    \begin{align}
  g(\psi, |\nabla \psi|) \Delta \psi + f(\psi, |\nabla \psi|) &= 0 \qquad \qquad \ \
  \  \text{\rm  in } \Omega,\\
   \psi &= 0\qquad \qquad \ \  \text{ \rm on } \pa \Omega,\\
   \partial_n \psi &= {\rm (const.)} \qquad \text{ \rm on } \pa \Omega.
  \end{align}
If $\psi > 0$ or
  $\psi < 0$ in $\Omega$
  then $\Omega$ is a ball and $\psi$ is radially symmetric.
\end{thm}

As a result, we expose the following rigidity of solutions of the free boundary problem
\begin{thm}\label{rigthm2d}
Let $k\geq 4$ and  $\Omega$ be a  simply connected  bounded domain. Let $u=\nabla^\perp \psi$, $B=\nabla^\perp A$ and $\Omega$ be a $C^k$ steady free boundary solution of the system \eqref{mhd1}--\eqref{mhd7} with trivial external field and current  $B_{\mathsf{ext}} =j_{\mathsf{ext}} =0$. Suppose that $B$ is non-vanishing except at a single point with $k$-times differentiable travel time.  Suppose additionally that $|B|\neq |u|$ away from the null point $p$ of $B$ and that $\lim_{x\to p} |u(x)|/|B(x)|\neq 1$. Then $\Omega= B_R(0)$ with $R$ such that ${\rm vol}(B_R(0))= {\rm vol}(\Omega_0)$ and the velocity and magnetic field are circular  $($i.e. $A:= A(r)$ and $\psi:= \psi(r)$$)$.
\end{thm}

    \begin{figure}[h!]
      \includegraphics[height=.3 \linewidth]{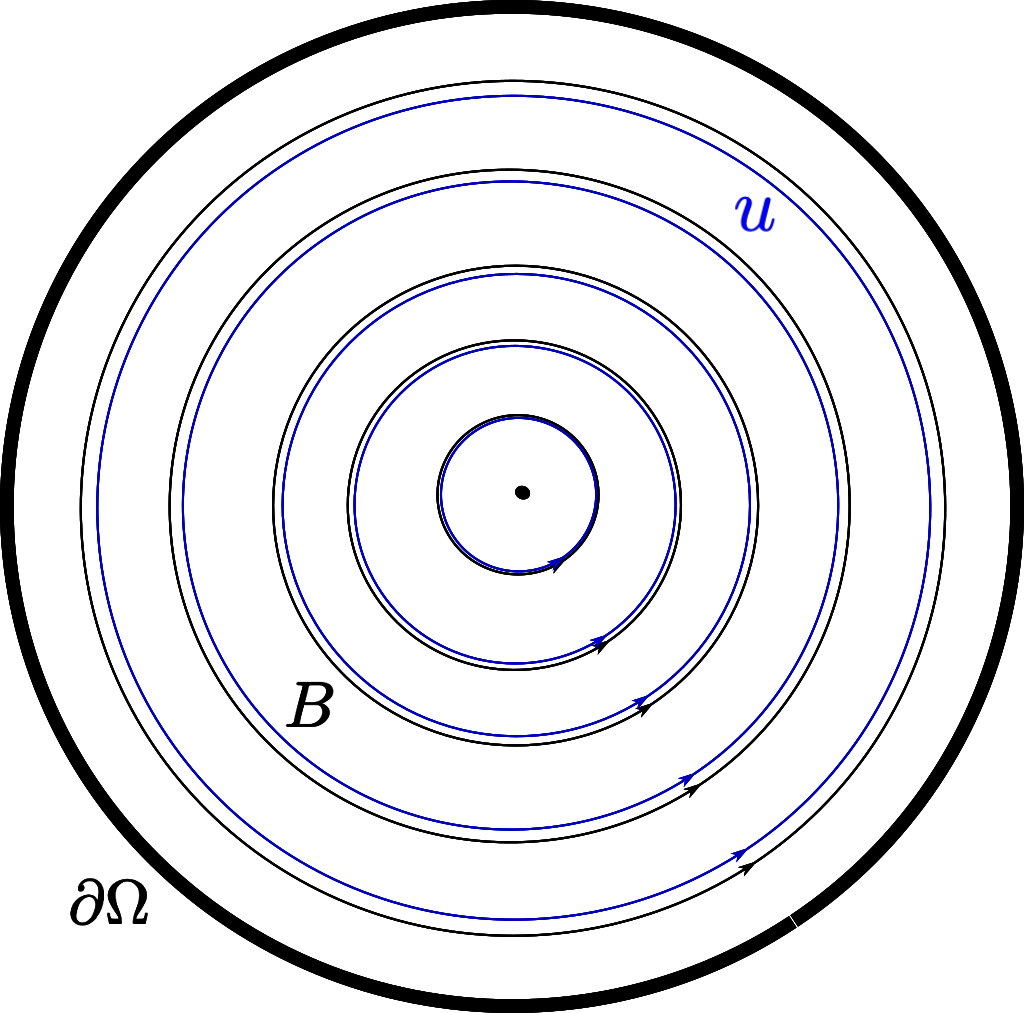}
    \end{figure}

\begin{rem}[Magnetohydrostatic equilibria]
In the special case of $u=0$ ($G=0$),  we deduce that the self-supporting magnetic field must be radial. This result was obtained recently by Hamel and Nadirashvili \cite{HN21} (Theorem 1.10 therein) and is meaningful in the context of Grad's conjecture concerning rigidity of magnetohydrostatic equilibria, see discussion in \cite{CDG21a,CDG21}.
\end{rem}

\begin{proof}[Proof of Theorem \ref{rigthm2d}] Our assumptions put us in the setting of Corollary \ref{corrol}.
Note that, by our assumption that the velocity and magnetic field are never commensurate, we have the existence of a $c_0>0$ such that $\left|1 - |G'(A)|^2\right|\geq c_0$.
Since, by assumption, $F,G''\in C^2$, $A = 0$ on the
boundary and $|\nabla A|$ vanishes only at one point in the interior, it follows that $A$ has a single sign in $\Omega$
 we are in the setting of Theorem \ref{serrin}. Applying this completes the proof.
\end{proof}

\begin{rem}[Breaking symmetry using flow velocity]
We remark that, for any given flux function $A$,
if $G'(A)^2 = 1$ identically then equation \eqref{Gpsieqn}
is a tautology (with $F=0$) and \eqref{psieqn3} is satisfied with $H(0)= 0$ and $B_{\mathsf{ext}}
=0$.  This corresponds taking $u=B$ where $B$ is chosen to be an arbitrary divergence free vector field tangent to a closed curve.  In this way, one can generate a free boundary steady state  $(u,B)$ on any desired domain.
\end{rem}

\subsection{Breaking symmetry using external fields} We now show how an external magnetic field can be used to shape the plasma in any desired configuration. Outside of the plasma body, there is a given steady
magnetic field in this region generated by an imposed current distribution $J_{\mathsf{ext}} $.
Despite  \eqref{psieqn}--\eqref{psieqn3} being an  overdetermined system, it can be hoped that by varying a designer external field $B_{\mathsf{ext}} $, one can obtain solutions occupying any desired geometry.  Indeed, it is this expectation that fuels hope in the scientific program surrounding stellarator confinement fusion.

We now show that this expectation is indeed the case here by constructing such free boundary steady states on \emph{any given domain} by designing a special imposed current.   The strategy is as follows.  Fix a simply connected domain $\Omega$, along with smooth functions $G$, $F$ and $H$ such that the elliptic problem  \eqref{psieqn}-\eqref{psieqn2} is solvable. The additional condition \eqref{psieqn3} will be fixed after the fact by designing the external field.
 The external field can be represented by $  B_{\mathsf{ext}}  = \nabla^\perp A_{\mathsf{ext}} $ where
\begin{align}
\Delta A_{\mathsf{ext}} &= J_{\mathsf{ext}}  \quad\quad  \text{ in } \mathbb{R}^2 \setminus \Omega,\\
A_{\mathsf{ext}} &=0 \quad\quad\quad \text{ on } \partial \Omega,\\
A_{\mathsf{ext}} &\to 0 \quad\quad\ \  \text{ as } |x|\to \infty.
\end{align}
Condition \eqref{psieqn3}  becomes an additional Neumann condition on the potential:
\be
 |\partial_n A_{\mathsf{ext}} |^2 = H(0)  +  \tfrac{1}{2}(1-|G'(0)|^2) |\partial_nA|^2,
 \label{additional}
\ee
and we would therefore like to choose $J_{\mathsf{ext}}$ so that given $A$, the restulting
potential $A_{\mathsf{ext}}$ satisfies \eqref{additional}.
Construction of a solution $ A_{\mathsf{ext}} $ of the above by an appropriate choice of the current  $ J_{\mathsf{ext}} $ leads to a simple inverse problem.  We solve this problem in the class of singular current sheet distributions, which provide a simple model for thin but densely packed coils surrounding the plasma body.
\begin{lemma}\label{lemcoils}
Let $\Omega$ be any simply connected domain with real analytic boundary and $f:\partial \Omega \to \mathbb{R}$ be a given analytic function with $f>0$.  Then, there exists a smooth Jordan curve $\Gamma$ surrounding $\Omega$ and a scalar function $j: \Gamma \to \mathbb{R}$  such that there is a solution $A_{\mathsf{ext}}:\mathbb{R}^2 \setminus \Omega\to \mathbb{R}$ to
\begin{align}\label{oda1}
\Delta A_{\mathsf{ext}} &=  j \delta_\Gamma  \quad\quad  \text{ \rm in } \mathbb{R}^2 \setminus \Omega,\\\label{oda2}
A_{\mathsf{ext}} &=0 \quad\quad\quad \  \text{\rm on } \partial \Omega,\\ \label{oda3}
\partial_n A_{\mathsf{ext}} &=f \quad\quad\quad \ \text{\rm on } \partial \Omega,\\
\nabla A_{\mathsf{ext}} &\to 0 \quad\quad\ \ \  \ \text{\rm as } |x|\to \infty. \label{oda4}
\end{align}
\end{lemma}
We offer two proofs of this lemma.  The first uses the Cauchy–Kovalevskaya theorem as in \cite{ELPS21}.  The second is more explicit and relies on a representation formula for solutions of Poisson problems in exterior domains.  These approaches have the possible drawback that the curve $\Gamma$ constructed by the argument must be taken close to the domain.
 We defer the details of these arguments to Appendix \ref{appendcoil}.

The assumption that $f$ is analytic is essentially sharp. In the simple case
that $\Omega$ is taken to be the unit disk, if the problem \eqref{oda1}-\eqref{oda4}
has a solution for some $C^1$ curve $\Gamma$ and $L^1(\Gamma)$ function $j$,
then the Fourier coefficients of $f$ need to decay exponentially fast
and this implies analyticity. See Lemma \ref{sharpness}. A more general result that holds for analytic domains $\Omega$ and allows for penetrative fields is proved in a forthcoming work of \cite{GRR21}.

With Lemma \ref{lemcoils} in hand, we  prove

\begin{thm}\label{nonsymmthmext}
Let $\Omega\subset \mathbb{R}^2$ be any simply connected domain with real analytic boundary.  Let $u=\nabla^\perp \psi$, $B=\nabla^\perp A$ and $\Omega$ be an analytic fixed boundary steady MHD solution  (e.g. solution of the system \eqref{mhd1}--\eqref{mhd7} without the free boundary condition \eqref{mhd5}) with $B$ non-vanishing except at a single point and such that
\be\label{condBvP}
|B|^2 >-\tfrac{2H(0)}{1-|G'(0)|^2} \qquad  \text{\rm on} \ \ \partial \Omega,
\ee
where $G$ and $H$ are defined in Proposition \ref{structuretheorem}.
Then there exists an external magnetic field generated by a singular charge distribution (current sheet) making $(u,B, \Omega)$ into a free boundary equilibrium.
\end{thm}

    \begin{figure}[h!]
      \includegraphics[height=.4 \linewidth]{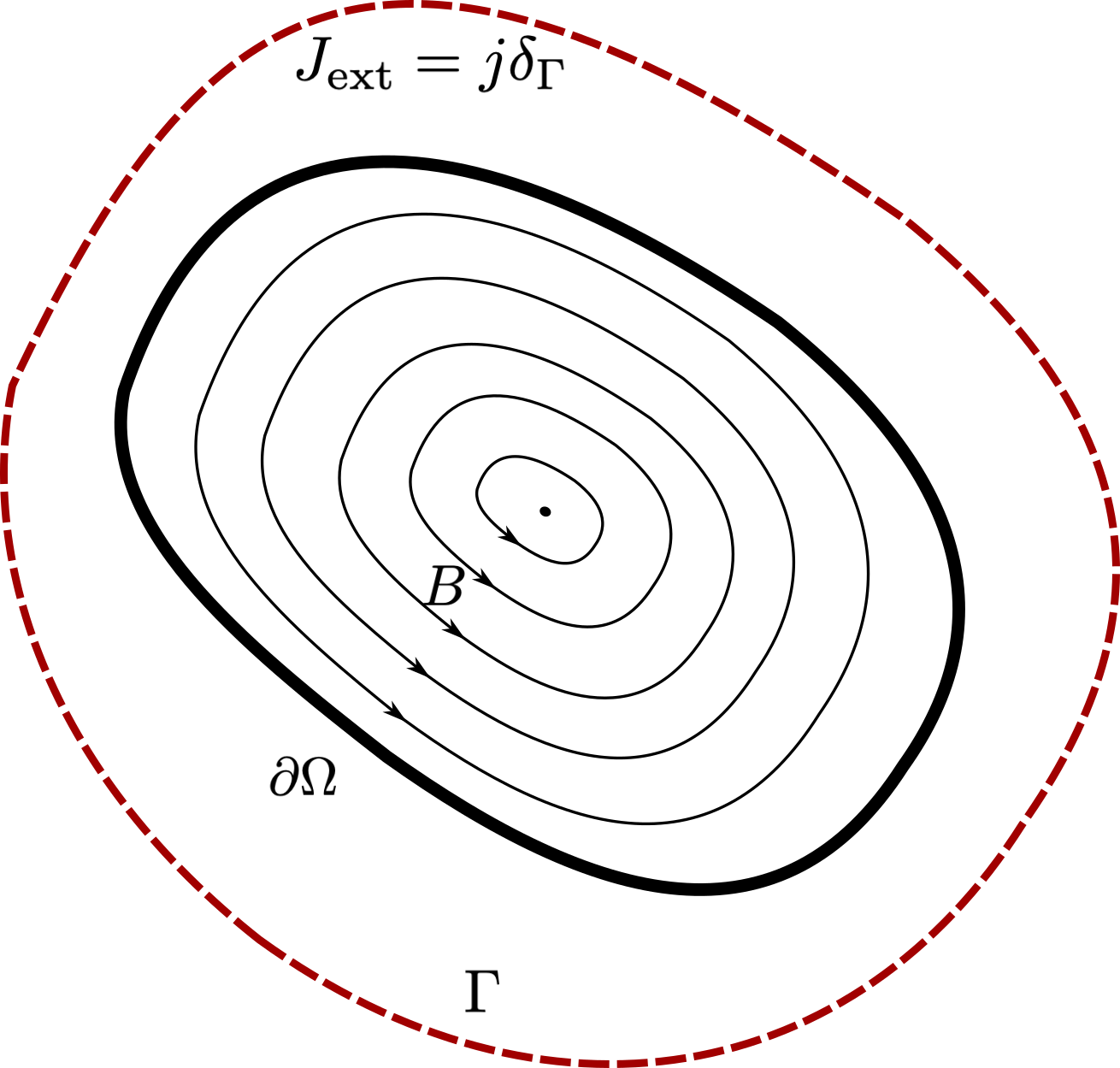}
    \end{figure}

   \begin{rem}[Structural assumption]
We first remark that the condition \eqref{condBvP} is automatic for any solution with positive plasma pressure $H>0$ having the additional  property that $|u|<|B|$ on the boundary of the domain (in particular, for magnetohydrostatic solutions which have $u=0$).
\end{rem}
\begin{rem}[Structure of coils] We make a few remarks about the realizability of the external currents.
First, the singular nature of the external charge distribution is not necessary and it can be made smooth.
Second, in practice it may be desirable to prescribe a constant current density along the coil.  Imposing this requirement results in the more complicated problem of  finding an appropriate curve $\Gamma$ to realize this constraint.  We do not address this interesting issue here.
Finally, it may be desirable to have current localized to points (in 2d) or curves (in 3d). In this direction, heuristically one expects that the current sheet could be approximated by many point charges allowing one to  create free boundary configurations on a domain nearby some target by using a sufficiently large but finite collection of ``point coils":
\be
J_{\mathsf{ext}}  = \sum_{i=1}^N j(x_i) \delta_{x_i} \approx  j \delta_\Gamma.
\ee
Of course, this replacement incurs an error and to make the Cauchy datum \eqref{oda2}--\eqref{oda3} hold exactly for a nearby fixed boundary steady state requires a non-trivial argument.
\end{rem}

   \begin{rem}[Construction of MHS equilibria and ``2d stellarator program"]
The main result of \cite{CDG21a} provides a method of constructing fixed boundary solutions with the desired ``laminar" structure of magnetohydrostatics (MHS) on domains which are slight deformations of disks.  Specifically, consider an MHD solution with $u=0$ and $B$ satisfying  MHS on a given domain $\Omega\subset \mathbb{R}^2$
\begin{align}
 B\cdot \nabla B&=\nabla p \qquad \quad \text{in} \ \ \Omega,\\
B \cdot \hat{n}&=0  \qquad\quad \quad   \text{on} \ \ \partial \Omega.
\end{align}
Theorem 1.3 of \cite{CDG21a} guarantees that solutions of the above which satisfy the Arnold stability conditions are deformable to arbitrary close domains.  Their magnetic field line topology is preserved by this deformation.  If $B= \nabla^\perp A$ and $\eta:= \nabla^\perp \cdot B= \Delta A$, this concerns stationary solutions of the form $\eta = F(A)$ where $F:\mathbb{R}\to \mathbb{R}$ satisfies $F'(A) >   -\lambda_1$ where $\lambda_1:= \lambda_1(\Omega)$ is the first eigenvalue of the Dirichlet Laplacian on $\Omega$. We remark that, on a convex domain $\Omega$, all Arnold stable flows have a single stagnation point in the interior.  See Theorem 1.4 of \cite{N13}.  This result also holds for the  first Dirichlet eigenfunction of the Laplacian.  Note also that if $\Omega= B_1(0)$ is the disk, then any radial potential $A=A(r)$ defines a stationary solution and if $A(r)$ is monotone decreasing away from $r=0$ then the stability condition above is satisfied.  Provided the travel time of these solutions is bounded at their centers, Theorem 1.3 applies to deform any member of this large class of stationary solutions to nearby analytic domains.  Theorem \ref{nonsymmthmext} herein then applies and such solutions can be realized as free boundary equilibria supported by suitable external currents. As such, one may regard this as a complete version of the (static) stellarator problem in two-dimensions - to identify asymmetric equilibria and give a constructive prescription for how to support them with external coils.  In three dimensions, this static problem has been approached numerically \cite{HM86,L17} but there are issues related to the possible non-existence of suitable equilibria outside symmetry (Grad's conjecture).  If identified, the dynamical problem of driving the physical plasma system to such a configuration by designing a control procedure involving running time varying currents through fixed external coils is a major open issue, even in two dimensions.
\end{rem}

\section{Three dimensions with axial symmetry}
There are results analogous to those in the previous section in three space
dimensions for symmetric configurations.
In this setting it is convenient to rewrite \eqref{mhd1} as
\begin{align}
\omega\times u - J \times B  &= - \nabla P, \label{eul}
\end{align}
where $J=\curl B$ and where we have introduced the Bernoulli function
\begin{equation}
 P = p + \frac{1}{2} |u|^2 - \frac{1}{2} |B|^2.
 \label{Pdef}
\end{equation}
Note also that the induction equation  \eqref{mhd2} can be understood simply as Lie transport
\be
\mathcal{L}_B u=0
\ee
where $\mathcal{L}_B u= [B, u]$ denotes the Lie derivative.
With $(r, \phi, z$) the standard cylindrical
coordinates, we will consider solutions $(u, B)$ which are
either independent of $z$ or $\phi$, meaning that
\begin{equation}
 \mathcal{L}_{e_z} B
 =
 \mathcal{L}_{e_z} u = 0
 \qquad
 \text{ or }
 \qquad
 \mathcal{L}_{e_\phi} B
=
\mathcal{L}_{e_\phi} u = 0
\end{equation}
with $e_z, e_\phi$ the usual basis vectors. Such solutions may occupy infinite cylinders or compact tori wrapping around the $z$-axis.

\subsection{$z$-independent solutions}
Here we consider solutions which satisfy
$\mathcal{L}_{e_z} B = \mathcal{L}_{e_z} u = 0$.
As such, these solutions occupy infinite cylindrical domains and contain the 2d setting as a special case. By
the elementary identity
\begin{equation}
 \curl (X \times Y) = X \div Y - Y \div X
 +\mathcal{L}_Y X,
 \label{ident}
\end{equation}
and $\div B = \div u = 0$ we have
\begin{equation}
 \mathcal{L}_{e_z} B = \curl (e_z \times B).
\end{equation}
We therefore make the following ansatz for $B$,
\begin{equation}
 B = C(A) e_z + e_z \times \nabla A,
 \qquad
 A = A(r,\theta),
 \label{Bansatz}
\end{equation}
and similarly for $u$ we take
\begin{equation}
 u = F(\psi) e_z + e_z\times \nabla \psi, \qquad
 \psi = \psi(r,\theta).
 \label{uansatz}
\end{equation}
Here $C, F$ are arbitrary functions. The vector fields in \eqref{Bansatz}-\eqref{uansatz}
are automatically divergence free.

The induction equation is
\begin{equation}
 \mathcal{L}_B u = 0.
\end{equation}
Using \eqref{ident} and that $\div u =\div B = 0$,
 we have $
 \mathcal{L}_B u = \curl (u\times B)
$
and consequently
\begin{align}
u\times B\nonumber
 &= C(A)\nabla \psi - F(\psi)\nabla A
 +\left( e_z \cdot (\nabla \psi \times \nabla A)\right)e_z.
\end{align}
If we take $\psi = G(A)$ for a function
$G$ this last term vanishes. Taking the curl we then find
\begin{equation}
 \curl ( C(A) \nabla \psi - F(\psi) \nabla A) =
 \left( C'(A) + F'(\psi) \right) \nabla A\times \nabla \psi
 =0.
\end{equation}
We therefore see that if we take $B, u$
as in \eqref{Bansatz}-\eqref{uansatz} with
$\psi = G(A)$ the induction equation holds automatically.
We now consider the momentum equation \eqref{eul}. Using
\begin{equation}
 \curl (C(A) e_z) = C'(A)\nabla A \times e_z,
 \qquad
 \curl (\nabla A \times e_z)
 = -\Delta A \, e_z,
\end{equation}
we find
\begin{equation}
 B \times \curl B
 =(C(A)e_z + e_z\times \nabla A) \times (C'(A) \nabla A  \times
 e_z + \Delta A e_z)
 = \left( CC'(A) + \Delta A\right)\nabla A,
\end{equation}
and in the same way
\begin{equation}
 u\times \curl\, u
 = (FF'(\psi) + \Delta \psi )\nabla \psi,
\end{equation}
and the momentum equation \eqref{eul} reduces to
\begin{equation}
   \left( FF'(\psi) + \Delta \psi\right)\nabla \psi-\left( CC'(A) + \Delta A\right)
    \nabla A = -\nabla P.
\end{equation}
Provided $\nabla A$ is non-vanishing except at possibly a single
point with $k$-times differentiable travel time this gives $P = P(A)$, and we then
find the Grad-Shafranov equation
\begin{equation}
  (1 - |G'(A)|^2)\Delta A - G''(A)G'(A)|\nabla A|^2
  + FF'(A) - CC'(A) = -P'(A).
\end{equation}
abusing notation and writing $FF'(A) =
 FF'(G(A))$.
 We then have an analog of Corollary \ref{corrol}.
\begin{lemma}\label{corrol2}
Define $B, u$ as in \eqref{Bansatz}-\eqref{uansatz} where $\psi = G(A)$
for a function $G$ and where $\nabla A$ is non-vanishing except at possibly a single
point with $k$-times differentiable travel time. Then $u, B, \Omega \subset \{z = 0\}$ form
a $C^k$ steady solution to the free boundary problem \eqref{mhd1}-\eqref{mhd7}
with $B_{\mathsf{ext}} \cdot \hat{n} = 0$ provided $A$ solves
\begin{alignat}{2}
  (1 - |G'(A)|^2)\Delta A - G''(A)G'(A)|\nabla A|^2
  + FF'(A) - CC'(A) &= -P'(A),
 &&\quad \text{\rm in }  \Omega,
  \label{A1}\\
  A&=0, &&\quad  \text{\rm on }  \partial \Omega,\\ \label{A3}
   \tfrac{1}{2}(|G'(0)|^2-1) |\partial_nA|^2&=H(0) -  \tfrac{1}{2}|B_{\mathsf{ext}} |^2,
    &&\quad
     \text{\rm on } \partial \Omega,
\end{alignat}
where $H(A) = P(A) - \frac{1}{2} F(G(A))^2 + \frac{1}{2}C(A)^2$.
\end{lemma}

Lemma \ref{corrol2} puts us again in the setting of Serrin \cite{S71} and so just as in Theorem
\ref{rigthm2d} we have
\begin{thm}\label{3dthm1}
Let $k\geq 4$ and let $B, u, A, \Omega$ be as in  Lemma \ref{corrol2},  a $C^k$ steady free boundary solution of the system \eqref{mhd1}--\eqref{mhd7} with trivial external field and current  $B_{\mathsf{ext}} =j_{\mathsf{ext}} =0$. Suppose that $B$ is non-vanishing except at a single point with $k$-times differentiable travel time.  Suppose additionally that $|B|\neq |u|$ away from the null point $p$ of $B$ and that $\lim_{x\to p} |u(x)|/|B(x)|\neq 1$.   Then $\Omega= B_R(0)$ with $R$ such that ${\rm vol}(B_R(0))= {\rm vol}(\Omega_0)$ and the velocity and magnetic field are circular  (i.e. $A:= A(r)$ and $\psi:= \psi(r)$).
\end{thm}

\subsection{$\phi$-independent solutions}\label{phiind}
These solutions can either occupy cylindrical domains or toroidal domains which wrap around the $z$-axis (e.g. the Tokamak). In the first case the equilibrium occupies an unbounded
domain in $\R^3$ but in the second case it occupies a bounded domain,
 Theorem \ref{serrin} can still be used to get a rigidity statement,
but it is far weaker than the rigidity in the cylindrical or 2D case. To
get this result, we make the ansatz
\begin{equation}
 B = \frac{1}{r} C(A) e_\phi + \frac{1}{r}
 e_\phi \times\nabla A,
 \qquad A := A(r,z),
\end{equation}
and
\begin{equation}
 u = \frac{1}{r} F(\psi) e_\phi + \frac{1}{r}e_\phi
 \times \nabla \psi,
  \qquad \psi := \psi(r,z).
\end{equation}
These are both divergence-free and axisymmetric, $\mathcal{L}_{r e_\phi}B
= \mathcal{L}_{re_\phi} u = 0$. We also have
\begin{equation}
  u \times B =
  \frac{C(A)}{r^2} \nabla \psi - \frac{F(\psi)}{r^2}
   \nabla A
  + \frac{1}{r}\left( e_\phi\cdot(\nabla \psi \times \nabla A)\right) e_\phi.
\end{equation}
Taking $\psi = G(A)$ as in the
$z$-independent case the last term drops out. Taking the curl, we get
\begin{multline}
 \curl(u\times B)
 = \nabla \left(\frac{C(A)}{r^2}\right)\times \nabla \psi
 - \nabla \left(\frac{F(\psi)}{r^2}\right)\times \nabla A
 \\
 =
 \frac{1}{r^2} \left( C'(A) + F'(\psi)\right) \nabla A  \times \nabla
 \psi +
 \frac{1}{r^3} \left( F(\psi) \nabla r \times \nabla A -  C(A)
 \nabla r \times \nabla \psi\right).
\end{multline}
The first term vanishes when $\psi = G(A)$.
To get the second term to vanish
we must take additionally
\begin{equation}
 F(G(A)) - C(A) G'(A) = 0.
\end{equation}
In other words, thinking of $F, C$ as given, to
satisfy the induction equation, $G$ must be chosen so
\begin{equation}
 \frac{G'(A)}{F(G(A))} = C(A).
 \label{parametersrequirement}
\end{equation}
As for the momentum equation \eqref{eul}, we compute
\begin{equation}
 \frac{1}{r} e_\phi \times \nabla A
 = \frac{1}{r} \pa_z A e_r - \frac{1}{r} \pa_r A e_z,
\end{equation}
and
\begin{equation}
 \curl \left( \frac{1}{r} e_\phi \times \nabla A\right)
 = \left(\pa_r\left( \frac{\pa_r A}{r}\right) +
 \frac{1}{r}\pa_z^2 A\right) e_\phi
 = \Delta^* A \frac{e_\phi}{r},
\end{equation}
where we have introduced the Grad-Shafranov operator
$
 \Delta^* = \pa_r^2 + \pa_z^2 - \frac{1}{r}\pa_r.
$
Since $\curl\, \left(\frac{1}{r} e_\phi\right) =  0$,
\begin{equation}
 \curl B = C'(A) \nabla A \times \frac{e_\phi}{r}
 - \Delta^* A \frac{e_\phi}{r}.
\end{equation}
Performing a similar calculation for $u$ we find
\begin{equation}
 B \times \curl B = \left(\Delta^*A + CC'(A) \right)
 \frac{\nabla A}{r^2}, \qquad
 u \times \curl u
 = \left(\Delta^* \psi + FF'(\psi)\right)
 \frac{\nabla \psi}{r^2}.
\end{equation}
As a result, the momentum equation reads
\begin{equation}
  -\left( \Delta^* A + FF'(A)\right) \nabla A
  +\left(\Delta^*\psi + CC'(\psi)\right)\nabla \psi
  =-r^2\nabla P,
\end{equation}
which, just as in the $z$-indepedent case, leads to
the Grad-Shafranov equation
\begin{equation}
 (1 - |G'(A)|^2)\Delta^* A
 - G''(A) G'(A) |\nabla A|^2 + CC'(A) - FF'(A)
 = r^2 P'(A).
\end{equation}
As a consequence we have
\begin{lemma}\label{corrol3}
Define $B, u$ as in \eqref{Bansatz}-\eqref{uansatz} where $\psi = G(A)$
for a function $G$ and take $P = P(A)$. Then $u, B, \Omega$ form
a $C^k$ steady solution to the free boundary problem \eqref{mhd1}-\eqref{mhd7}
with $B_{\mathsf{ext}} \cdot \hat{n} = 0$ provided $G$ is chosen to satisfy
\begin{equation}
G'(A) = C(A)F(G(A)),
\end{equation}
and $A$ is a solution to
\begin{alignat}{2}
  (1 - |G'(A)|^2)\Delta^* A - G''(A)G'(A)|\nabla A|^2
  + CC'(A) - FF'(A) &= r^2P'(A), \quad &&\text{\rm in } \Omega, \label{A1phi}\\
  \label{A2phi}
  A&=0,   \quad  &&\text{\rm on }  \partial \Omega,\\
   \label{A3phi}
   \tfrac{1}{2r^2}(|G'(0)|^2-1) |\partial_nA|^2&=H(0,r) -  \tfrac{1}{2}|B_{\mathsf{ext}} |^2,
    \quad &&\text{\rm on } \partial \Omega,
\end{alignat}
where $H(A, r) = P(A) - \frac{1}{2r^2} F(G(A))^2 + \frac{1}{2r^2} C(A)^2$.
\end{lemma}

In the case of $u=0$, this result appears as Lemma 1.1 in \cite{DEPS21}.
The above equation and boundary condition are not invariant under arbitrary reflections and so there is not such a strong
rigidity result in this setting. They are however invariant under reflections
$z \mapsto -z$ and so the arguments from \cite{S71} give
\begin{thm}\label{3dthm2}
Let $k\geq 4$ and let $B, u, A, \Omega$ be as in Lemma \ref{corrol3},  a $C^k$ steady free boundary solution of the system \eqref{mhd1}--\eqref{mhd7} with trivial external field and current  $B_{\mathsf{ext}} =j_{\mathsf{ext}} =0$. Suppose additionally that $B\neq u$ along any magnetic field line.
 Then $\Omega$, the velocity, and magnetic field are symmetric with
 respect to the reflection $z\mapsto -z$.
\end{thm}
This theorem restricts the shape of the tori in which a free boundary equilibrium without external forcing can possibly reside.  As in our 2D construction, external coils could be used to impose external fields to bend these equilibria at will.  We remark that Theorem 1.2 of the recent work \cite{DEPS21} constructs solutions of the overdetermined problem \eqref{A1phi}--\eqref{A3phi}, resulting in axisymmetric free boundary Euler solutions exhibiting this up-down symmetry.
We remark that breaking of up-down symmetry plays an important role in enhancing turbulent transport in Tokamaks \cite{B16,PBP11}.

\subsection{Breaking symmetry using external fields}

Here we show how a current sheet enveloping the toroidal plasma can be designed to produce free boundary equilibria.
We have the following Lemma
\begin{lemma}\label{lemcoils2}
Let $\Omega\subset \mathbb{R}^2$ be any simply connected domain  with real analytic boundary which corresponds to  the cross-section in the $\{\phi=0\}$ plane of an axisymmetric toroidal domain $T$.
Let $f:\partial \Omega \to \mathbb{R}$ (extended as a $\phi$-independent function on $T$) be a given analytic function with $f>0$.  Then, there exists a toroidal surface $T'$ with cross-section $\Omega'$ surrounding $T$ and a scalar function $j: \Gamma \to \mathbb{R}$  such that there is a vector potential $A_{\mathsf{ext}}:\mathbb{R}^2 \setminus T\to \mathbb{R}^3$ to
\begin{alignat}{2}\label{oda12}
\Delta A_{\mathsf{ext}} &=  j e_\phi \delta_{T'}\qquad  &&\text{\rm in } \mathbb{R}^3 \setminus T,\\
\hat{n}\cdot \curl A_{\mathsf{ext}} &=0 &&   \text{\rm on } \partial T,\\
|\hat{\tau}\cdot\curl A_{\mathsf{ext}}| &=f  && \text{\rm on } \partial T,\\
e_\phi\cdot\curl A_{\mathsf{ext}} &=0  &&  \text{\rm on } \partial T,\\
\nabla A_{\mathsf{ext}} &\to 0  && \text{\rm as } |x|\to \infty, \label{oda42}
\end{alignat}
where $\hat{n}$ is a normal vector field to $\Omega$ in the $\{\phi = 0\}$
half-plane and $\hat{\tau}$ is a unit tangent vector tangent to $\Omega$ in the $\{\phi = 0\}$
half-plane.
\end{lemma}
The proof (using Cauchy–Kovalevskaya) of this lemma is similar to that of Lemma \ref{lemcoils} due to the assumption of axisymmetry.
With it, we obtain

\begin{thm}\label{nonsymmthmext2}
Let $\Omega\subset \mathbb{R}^2$ be any simply connected domain  with real analytic boundary which corresponds to  the cross-section in the $\{\phi=0\}$ plane of an axisymmetric toroidal domain $T$.   Let $u$, $B$ and $T$ be an axisymmetric analytic fixed boundary steady MHD solution of the type described in Lemma \ref{corrol3} and such that
\be
|\partial_n A|^2 >-\tfrac{2H(0)}{1-|G'(0)|^2} \qquad  \text{\rm on} \ \ \partial \Omega,
\ee
where $n$ is the normal to $\Omega$ in the  $\{\phi=0\}$ plane and where $A$, $G$, and $H$ are defined in Lemma \ref{corrol3}.
Then there exists an external magnetic field generated by a singular charge distribution (toroidal current sheet) making $(u,B, T)$ into a free boundary equilibrium.
\end{thm}

\appendix

\section{Reconstruction Lemma}
\begin{lemma}\label{structuretheorem1}
Let $ \psi$ be $C^{k+1}(\Omega)$.  Suppose that $|\nabla \psi|$ is non-vanishing except at a single point in $\Omega$, occurring on the level $\{\psi= c_*\}$. Suppose moreover that the particle travel time satisfies
\be
\label{traveltimebd}
\mu(c):= \oint_{\{ \psi=c\}}  \frac{\rmd \ell}{|\nabla \psi|} \in C^k( {\rm rang}(\psi)),
\ee
 and that $w$ is a $C^{k}(\Omega)$ solution of
\be\label{transp}
\nabla^\perp \psi \cdot \nabla w = 0.
\ee
 Then, there exists a $C^{k}$ function $F:\mathbb{R}\to \mathbb{R}$ such that
\be
w = F(\psi).
\ee
\end{lemma}
\begin{rem}[Travel time for radial streamfunctions]\label{traveltimerem}
To understand better the assumption \eqref{traveltimebd}, we discuss the case where the domain is a disk and the streamfunction is radial $\psi(x)= f(|x|)$ for $f$ invertible. Without loss of generality, we identify the stagnation point as the set $\{\psi=0\}$. Then
\be
\mu(c)= \oint_{\{ r=f^{-1}(c)\}}  \frac{\rmd \ell}{|f'(r)|} =  \frac{ 2\pi f^{-1}(c)}{|f'(f^{-1}(c))|} = \pi \frac{\rmd}{\rmd c} |f^{-1}(c)|^2.
\ee
Thus, the condition \eqref{traveltimebd} is that $(f^{-1})^2\in C^{k+1}( {\rm rang}(\psi))$. If $f(r) = r^n$, then this requirement is that $n\leq 2$.
The regularity assumption of $ \psi \in C^{k+1}(\Omega)$ then implies $n=2$.
More generally, if $f(r) = r^2 + g(r)$ a calculation shows that
if $h(r) = \frac{g'(r)}{r}$ is $C^k$ near $r = 0$, then $\mu \in C^{k+1}.$
\end{rem}

\begin{proof}[Proof of Lemma \ref{structuretheorem1}]
We introduce angle action coordinates $(x,y)\mapsto (\psi, \theta)$ as follows. Under our assumptions, the level sets $\{ \psi= c\}$ are Jordan curves.  This system allows for a canonical transformation to action-angle variables, $(x,y)\mapsto (\psi,\theta)$ where $\psi$ is the ``radial coordinate" and  the ``angular coordinate" is
\be
\theta(x)= \frac{2\pi}{\mu(\psi(x))} \int_{\Gamma_{x_0(\psi),x} } \frac{\rmd \ell}{|\nabla \psi|}, \qquad
\ee
where $\rmd \ell$ is the arc length differential.
 In the above, fixing a point $x\in \Omega$, the line integral is taken counterclockwise from a point $x_0(\psi)=x_0(\psi(x))$ (defined directly below) on the curve $\{ \psi=\psi(x)\}$ to the point $x$.   Here, $x_0$ is defined in the following way: fix a $y\in \partial \Omega$ and flow transversally to the  levels of $\psi$ via $\dot{Y}_s(y)= \nabla \psi( Y_s(y))$ with $Y_0(y)=y$. Since $\nabla \psi\neq 0$ except at a single point $p$, the curve $Y_s(y)$ crosses each level set of $\psi$ exactly once and tends to the stagnation point at infinite $s$.  Therefore, there exists a unique point $x_0(\psi)$ such that $Y_s(y)$ crosses the level set with value $\psi$. This choice of $x_0$ is differentiable in $\psi$. Then $\theta(x)$ is a $2\pi$--periodic parametrization of the streamline with value $\psi(x)$.
  In these coordinates, \eqref{transp} becomes
\be
\mu(\psi) \partial_\theta w=0,
\ee
whence we conclude that $w:= F(\psi)$ locally, for some function $F\in C^1$ (away from the stagnation point) and continuous on the whole range (see proof of Theorem 1.10 in \cite{HN21}).  To prove that $F$ can be defined globally, we use flow by the transverse direction $\nabla \psi$ from the boundary to the stagnation point $\{\psi= c_*\}$  using the fact that $|\nabla \psi|$ is non-zero except at $\{\psi= c_*\}$. See proof of Lemma 2.1 in \cite{CDG21a}. Away from the stagnation point, regularity of $F$ follows.  For global regularity, note  that
\begin{align}
\nabla 
&=  \tfrac{1}{|\nabla \psi|}\left( \widehat{\nabla^\perp \psi} \mu(\psi) \partial_\theta +  \widehat{\nabla \psi}\partial_\psi \right).
\end{align}
This fact, together with  $w\in C^{k}(\Omega)$ and  $\mu\in C^k( {\rm rang}(\psi))$, imply $F\in C^k( {\rm rang}(\psi))$.
\end{proof}

\section{The coil problem}\label{appendcoil}

We first present the proof by Cauchy–Kovalevskaya.
\begin{proof}[Proof of Lemma \ref{lemcoils}]\label{proofofLem}
Since $\partial \Omega$ and $f$ are both real analytic, by the Cauchy–Kovalevskaya  theorem there exists a collar neighborhood of $\partial \Omega$, call it $\mathcal{N}$ and a scalar function  $A_{\mathsf{H}}: \mathcal{N}\to \mathbb{R}$ such that
\begin{align}
\Delta A_{\mathsf{H}} &= 0 \quad\quad \ \ \  \text{ in } \mathcal{N},\\\label{dirchAh}
A_{\mathsf{H}} &=0 \quad\quad\quad \text{ on } \partial \Omega,\\ \label{normalderAh}
\partial_n A_{\mathsf{H}} &= f \quad\quad\quad \text{ on } \partial \Omega.
\end{align}
Since $f>0$, \eqref{dirchAh} and \eqref{normalderAh} imply that the level sets of $A_{\mathsf{H}}$ are closed Jordan curves surrounding $\partial \Omega$ in a possibly smaller collar neighborhood of the boundary (this follows from the fact that $f>0$ implies that there is a non-zero component of $\nabla A_{\mathsf{H}}$ near  $\partial\Omega$ and so it follows by the implicit function theorem that the level sets must be diffeomorphic to the boundary). These level curves foliate that neighborhood and $A_{\mathsf{H}}$ is strictly monotone across them. By possibly shrinking the neighborhood, we redefine $\mathcal{N}$ to be any such collar region. In light of this, consider now any $a_0\in \mathbb{R}$ such that the level curve $\Gamma := \{A_{\mathsf{H}} =a_0\}$ lies in $\mathcal{N}$.   Define $A_{\mathsf{ext}}: \mathbb{R}^2 \setminus \Omega \to \mathbb{R}$ to be
\be
A_{\mathsf{ext}}
=\begin{cases} \chi (A_{\mathsf{H}}(x)) & x\in \mathcal{N} \\ a_0  & x\in \mathcal{N}^c \setminus \Omega \end{cases},
\qquad \text{where} \qquad
\chi(z)= \begin{cases} z  & z\leq a_0 \\
a_0 & z>a_0
\end{cases}.
\ee
Since $A_{\mathsf{H}}$ is harmonic it is easy to check (and make rigorous by an approximation argument) that
\be
\Delta A_{\mathsf{ext}} = \begin{cases} |\nabla A_\mathsf{H}|^2\chi''(A_\mathsf{H})  & x\in \mathcal{N} \\ 0  & x\in \mathcal{N}^c \setminus \Omega \end{cases} =  -|\nabla A_{\mathsf{H}}|^2\delta_\Gamma,
\ee
in the sense of distributions. Thus $A_{\mathsf{ext}}$ satisfies \eqref{oda1}--\eqref{oda4} with $\Gamma := \{A_{\mathsf{H}} =a_0\}$, $j=- |\nabla A_{\mathsf{H}}|^2$.
\end{proof}
Now we provide result which is similar in spirit to Lemma \ref{lemcoils} and
which is proved using a different method which
has the advantage of giving an essentially explicit formula for the
solution. We first consider the case that
the domain $\Omega$ is the unit disk.
\begin{lemma}\label{lemcoils3}
Let $\Omega$ be the unit disk
and let $f:\partial \Omega \to \mathbb{R}$ be a given analytic function.
Let $\Gamma$ be a $C^1$ perturbation of a circle of radius $R$ for any
$R >1$.
Then there exists a scalar function $j: \Gamma \to \mathbb{R}$  such that there is a solution $A_{\mathsf{ext}}:\mathbb{R}^2 \setminus \Omega\to \mathbb{R}$ to
\begin{alignat}{2}\label{oda1p}
\Delta A_{\mathsf{ext}} &=  j \delta_\Gamma
&&\quad\quad  \text{ \rm in } \mathbb{R}^2 \setminus \Omega,\\
A_{\mathsf{ext}} &=0 &&\quad\quad \  \text{\rm on } \partial \Omega,\\
\partial_n A_{\mathsf{ext}} &=f &&\quad\quad \ \text{\rm on } \partial \Omega,
\label{oda3p}\\
\nabla A_{\mathsf{ext}} &\to 0 &&\quad\quad\  \text{\rm as } |x|\to \infty. \label{oda4p}
\end{alignat}
\end{lemma}
The difference between this and Lemma \ref{lemcoils} is that for this lemma,
we can prescribe the shape of the coil $\Gamma$ ahead of time
and in particular the coil can be taken arbitrarily far from the domain. The
tradeoff is that as stated, the result only applies
when $\Omega$ is the unit disk. We briefly discuss how it can be used to prove
an analog of Lemma \ref{lemcoils} which handles more general domains
$\Omega$ at the cost of no longer being able to fix the geometry of the
coils ahead of time.

If $\Omega$ is an arbitrary simply connected bounded domain, there is a conformal transformation
\be
\Phi :\Omega \to D
\la{phi}
\ee
mapping $\Omega$ to the unit disk.
Since the boundary of $\Omega$ is analytic, there is an open set $U$ with smooth boundary
$\pa U$ so that $\Phi$ can be extended to $\ov{U}$ so that the extension is
 $C^1$ up to the boundary and is conformal in $U$.
The image of $U$ under the extended conformal map, still denoted $\Phi$, $\Phi(U)$ is an open set containing $\ov{D}$.
We can then argue as above to construct a circle $\Gamma$ of radius $R>1$ and a solution $a$ in the conformally transformed plane, solving \eqref{oda1p}-\eqref{oda4p} there.
Composing with $\Phi^{-1}$ as above gives a solution to the problem \eqref{oda1}-\eqref{oda4}
which is defined only in $\ov{U}$. We can then argue as in the proof
of Lemma \ref{lemcoils} to extend this solution to $\R^2$.

\begin{proof}[Proof of Lemma \ref{lemcoils3}]
Fixing $j, \Gamma$ for the moment, we start by giving
an explicit representation formula for any $C^2$ solution to the exterior Poisson problem
\begin{align}
 \Delta a &= J \delta_{\Gamma} \qquad \text{ in } \mathbb{R}^2 \setminus \Omega,
 \label{apbm1}\\
 a|_{\pa \Omega} &= 0\qquad
\ \  \lim_{|x| \to \infty} \nabla a(x) = 0.
 \label{apbm3}
\end{align}
For this it is convenient to work in complex variables.
We parametrize the curve $\Gamma = \{z = \zeta(\theta)\}$ for a function $\zeta(\theta)$
and write $j(\theta) = J(\zeta(\theta))$.
Define the quantities
\begin{equation}
 c_0 = \frac{1}{2\pi}\int_0^{2\pi} j(\theta) \log |\zeta(\theta)| |\zeta'(\theta)|\rmd \theta,
 \label{c0}
\end{equation}
and for $k \not=0$,
\be
c_k = \fr{1}{2\pi}\int_0^{2\pi} \fr{1}{k} (\zeta(\theta))^{-k} j(\theta)|\zeta'(\theta)|\rmd \theta.
\la{ck}
\ee
Let us note for later use that with
$\min_{\theta \in [0,2\pi]} |\zeta(\theta)| = R$, for
$k\not = 0$, the coefficients $c_k$ satisfy
\begin{equation}
 |c_k| \leq C_{\Gamma, j} \frac{1}{|k|} \frac{1}{R^{k}},
 \label{ckbd}
\end{equation}
where
\begin{equation}
 C_{\Gamma, j} = \frac{1}{2\pi} \int_{0}^{2\pi} |j(\theta)| |\zeta'(\theta)| \rmd \theta.
 \label{CGj}
\end{equation}
Using the expansion
\begin{equation}
  \log|1-z| = \int_0^1\fr{\rmd}{\rmd t}\log|1-tz|\rmd t = - \mathfrak{Re}  \left (\int_0^1\fr{z}{1-tz}\rmd t\right ),
 \label{}
\end{equation}
we write the Newton potential applied to the right-hand side of \eqref{apbm1} in the
form
\begin{equation}
 N(z)
 = \frac{1}{2\pi}\oint_{\Gamma} \log |z-\zeta| J(\zeta) \rmd S(\zeta)
 = \frac{1}{2\pi}\int_{0}^{2\pi}\log |z- \zeta(\theta)| j(\theta)\,|\zeta'(\theta)| \rmd \theta
 = c_0 -  \mathfrak{Re}  \sum_{k =1}^{\infty} c_k z^k.
 \label{}
\end{equation}
{We remark that the last identity holds for $|z|< R$ while the first two hold for all $z\in \mathbb{C}$}.
We now compute the harmonic extension of the restriction of $N$ to $\pa \Omega
= \{z = e^{i\theta}\}$.
We first define
\begin{equation}
 g(\theta) = N(e^{i\theta}) = c_0 - \mathfrak{Re} \sum_{k = 1}^{\infty} c_k e^{ik\theta}
 = c_0 - \sum_{k\not=0} g_k e^{ik\theta},
 \label{}
\end{equation}
where the coefficients $g_k$ are given by
\begin{equation}
  g_k = \fr{1}{2}c_k, \; \text{for}\, k>0, \quad g_k = \fr{1}{2}{\ov{c_{-k}}}, \;\text{for} \, k<0.
  \la{gk}
\end{equation}
The harmonic extension of $g$ to the exterior of $\Omega$ is obtained by noting
that the function
\begin{equation}
  - \mathfrak{Re} \sum_{k=1}^{\infty} c_k\left(\ov{z}\right)^{-k} = -  \mathfrak{Re} \sum_{k=1}^{\infty}\ov{c_k}z^{-k} = - \mathfrak{Re}  F(z)
 \label{}
\end{equation}
where
\begin{equation}
  F(z) = \sum_{k =1}^\infty\overline{c}_kz^{-k}
 \label{}
\end{equation}
is holomorphic in the exterior and decays at infinity. It follows that the function
\begin{equation}
 G(z) = c_0 - \mathfrak{Re} F(z)
 \label{}
\end{equation}
is harmonic, is constant at infinity, and satisfies $G|_{\pa \Omega} = g$.
Let
\begin{equation}
 a(z) = c_0 -  \mathfrak{Re}  F(z) - N(z) + \frac{ {f}_{\mathsf{ave}}}{2\pi} \log |z|
 \label{repformula}
\end{equation}
where ${f}_{\mathsf{ave}}:= \oint_{\partial \Omega} f \rmd \ell$.  It follows that $a$ satisfies
\begin{equation}
 \Delta a(z) = -\oint \delta(z-\zeta) J(\zeta) \rmd S(\zeta)
 \label{}
\end{equation}
in $|z| > 1$ and
\begin{equation}
 a(e^{i\theta}) = 0.
 \label{}
\end{equation}
We now want to choose $j(\theta)$ so that with $a$ given in
\eqref{repformula}, $a$ satisfies the Neumann condition \eqref{oda3p}.
The external normal derivative at $\pa \Omega$ is $-\pa_r$ and using the
above formulas we compute
\begin{equation}
 -\pa_r a(z)|_{z = e^{i\theta}}
 =
 -  \mathfrak{Re}  \sum_{k = 1}^\infty k c_k e^{ik\theta}
 -  \mathfrak{Re}  \sum_{k = 1}^\infty k \overline{c}_k e^{-ik\theta}+ \frac{ {f}_{\mathsf{ave}}}{2\pi} ,
 \label{}
\end{equation}
where recall $c_k = c_k[j]$ is given in \eqref{ck}.
Letting $f_k$ denote the Fourier coefficients of $f$, the equation
to solve is
\begin{equation}
  c_k = \begin{cases}
  -\frac{1}{|k|} f_k\; \quad  \ \text{for}\, \quad k>0, \\
-\frac{1}{|k|}\bar{f}_{-k} \quad \text{for}\, \quad k<0.
  \end{cases}
 \label{ckeqn}
\end{equation}
When the curve
$\Gamma$ is a circle of radius $R > 1$ so that $\zeta(\theta) = R e^{i\theta}$
we compute directly
\begin{equation}
 c_k= \frac{1}{k}R^{1-|k|} j_k
 \label{}
\end{equation}
where $j_k$ denote the Fourier coefficients of $j$, $j = \sum j_k e^{ik\theta}$.
Provided the Fourier coefficients of $f$ satisfy the bound
\begin{equation}
 |f_k| \leq C_{\Gamma, j} |k| R^{-|k|}
 \label{}
\end{equation}
with $C_{\Gamma, j}$ as in \eqref{CGj}, the equation \eqref{ckeqn} can be solved
directly for $j_k$.

If instead $\Gamma$ is a perturbation of the circle
of radius $R$, then the equation for $j$ is
\begin{equation}
 \int_0^{2\pi} j(\theta) w_k(\theta) \rmd \theta
 = R^{|k|-1} k c_k,
 \qquad \text{ where } w_k(\theta) = \frac{|\zeta'(\theta)|}{\zeta(\theta)^k}.
 \label{}
\end{equation}
Under our hypotheses, $w_k$ is a perturbation of $e^{i(1-k)\theta}$.
Writing
\begin{equation}
 E_k(j) = \int_{0}^{2\pi} j(\theta) (e^{i(1-k)\theta} - w_k(\theta))\rmd \theta,
 \label{}
\end{equation}
 we would like to find $j$ satisfying the equation
\begin{equation}
 j -2\pi \sum_{k} E_k(j) e^{-ik\theta} = 2\pi \sum_{k}  R^{|k|-1} k c_k  e^{-ik\theta}.
 \label{}
\end{equation}
This can be solved by Neumann series provided the sequence of operators $\{E_k\}$
has sufficiently small $\ell^2$ norm, where $\|\{E_k\}\|_{\ell^2}^2 =
\sup_{\|j\|_{L^2} = 1} \sum_{k} \|E_k(j)\|^2_{L^2}$. This condition holds
provided $\Gamma$ is a sufficiently small perturbation of a circle.
\end{proof}

The above construction actually shows that the requirement that $f$ is
analytic is essentially sharp, in the sense that when $\Omega$ is the unit
disk, under only mild assumptions on $j$ and $\Gamma$, if
there is a solution to the above problem then the Fourier coefficients
of $f$ decay exponentially fast.
\begin{lemma}
  \label{sharpness}
 Let $\Omega$ be the unit disk and suppose that $\Gamma$ is parametrized by $z = \zeta(\theta) = R(\theta)e^{i\theta}$ where
 $R(\theta)$ is a periodic $C^1$ function with
 $R(\theta) > 1$. Suppose that $j$ satisfies
 \begin{equation}
\oint_{\Gamma}  |j|<\infty.
  \label{}
 \end{equation}
 Then the problem \eqref{oda1}-\eqref{oda4} does not admit a solution $A_{\mathsf{ext}}
 \in C^2(\mathbb{R}^2\setminus \Omega)$ unless the Fourier coefficients
 of $f$, $f_k = \frac{1}{2\pi} \int_0^{2\pi} f(\theta) e^{ik\theta}\, \rmd \theta$ satisfy
 \begin{equation}
  |f_k| \leq C (1+|k|) \left(\max_{\theta \in [0,2\pi]}R(\theta)\right)^{-|k|},
  \label{sharpbd}
 \end{equation}
 for some constant $C > 0$. Thus, $f$ must be real analytic.
\end{lemma}
\begin{proof}
  Following the above calculation we find that any
   $C^2$ solution $a$ to the problem \eqref{apbm1}-\eqref{apbm3} can be
   represented as in \eqref{repformula} and under our hypotheses
   the coefficients $c_k$ from \eqref{ck} satisfy the bound
   \eqref{ckbd}. The Fourier coefficients of $f$ are related to
   $c_k$ from \eqref{ckeqn} and it then follows from \eqref{ckbd}
   that $f_k$ satisfies \eqref{sharpbd}.
   The fact that \eqref{sharpbd} implies that $f$ is analytic
   is classical, see for example Exercise 4 from section 1.4 of \cite{K04}.
 \end{proof}

Finally, we sketch the construction of an toroidal current sheet to hold steady a $\phi$-independent axisymmetric plasma configuration.
\begin{proof}[Proof of Lemma \ref{lemcoils2}]
From the computations in \S \ref{phiind}, since the domain $T$ and the data $f$ are axisymmetric, we seek an external field $B_{\mathsf{ext}}= \curl  \ A_{\mathsf{ext}} $ of the form
\be\label{ansatzexternal}
B_{\mathsf{ext}}  = \tfrac{1}{r} e_\phi \times \nabla a_{\mathsf{ext}}
\ee
for a scalar function $a_{\mathsf{ext}}:\mathbb{R}^3 \setminus T\to \mathbb{R}$ which satisfies $a:=a(r,z)$.  If $a_{\mathsf{ext}}$ is constant on the boundary $\partial\Omega$ then $\nabla a_{\mathsf{ext}} =|\nabla a_{\mathsf{ext}}| \hat{n}$ on the boundary so  the field is non-penetrative $B_{\mathsf{ext}}\cdot\hat{n}=0$.   Moreover, it follows that $ B_{\mathsf{ext}}= \frac{1}{r}|\nabla a_{\mathsf{ext}}| \hat{\tau}$ on the boundary
where recall $\hat{\tau}$ is the unit tangent vector.

The ansatz \ref{ansatzexternal} is automatically divergence-free. Its curl can be computed as
\be
\Delta A_{\mathsf{ext}}  = \tfrac{1}{r}\Delta^*a_{\mathsf{ext}} e_\phi.
\ee
where we recall the Grad-Shafranov operator
$
 \Delta^* = \pa_r^2 + \pa_z^2 - \frac{1}{r}\pa_r.
$
Thus to produce a solution of \eqref{oda12}--\eqref{oda42}, we must to solve the following planar (in $r-z$) problem
\begin{align}\label{oda111}
 \tfrac{1}{r}\Delta^* a_{\mathsf{ext}} &=  j  \delta_{\partial \Omega'}  \quad \ \ \ \text{ \rm in } \mathbb{R}^2 \setminus \Omega,\\
 a_{\mathsf{ext}}&=0 \quad\quad\quad \ \ \    \text{\rm on } \partial \Omega,\\
 \tfrac{1}{r}\partial_n  a_{\mathsf{ext}} &=  f \quad\quad\quad \  \ \ \text{\rm on } \partial \Omega,\\
\nabla  a_{\mathsf{ext}} &\to 0 \quad\quad\ \ \  \ \ \ \text{\rm as } |(r,z)|\to \infty. \label{oda411}
\end{align}
The remainder of the argument follows exactly as in Lemma \ref{lemcoils} (by the Cauchy–Kovalevskaya argument) since the  difference between the Laplacian and the  Grad-Shafranov operator is a first order term. Thus Cauchy–Kovalevskaya can be used to obtain a solution of $\Delta^* a_{\mathsf{GS}} =0$ in a collar neighborhood of $\partial \Omega$. Cutting  $a_{\mathsf{GS}}$ on  a level set (which again is diffeomorphic to $\partial \Omega$ by the implicit function theorem) gives  a solution $a_{\mathsf{ext}}$ on the whole plane with the addition of a current sheet force.  Flowing said level set by $r e_\phi$ yields the toroidal surface $T'$ in the statement.
\end{proof}
\begin{rem}
We note that the operator $ \tfrac{1}{r}\Delta^*$ has an explicit Green's function
\be
G(r,z;r',z'):= \frac{\sqrt{r'r}}{\pi k} \left[\left(1- \frac{k^2}{2}\right) K(k^2) - E(k^2)\right] \quad\text{where} \quad k^2 = \frac{4r' r}{(r+r')^2+ (z-z')^2},
\ee
where where $K(k)$ and $E(k)$ are the complete elliptic integrals of
the first and second kinds, respectively. See e.g. Appendix A of \cite{AP08}. In principle, one could use this explicit formula to solve \eqref{oda1}--\eqref{oda4} with  a singular current distribution as in the proof of Lemma \ref{lemcoils3}.
\end{rem}

  \section{The virial theorem}
  \label{virial}

We include, for the sake of completeness, what is known as the ``virial theorem'' in the plasma physics literature (see \S 4.3 of \cite{F14}). Similar considerations have been used to rigorously establish non-existence results for fluid equilibria with suitable decay at infinity \cite{C10,CC15}.
  Here we consider functions  $(\rho,u, q, B)$ in a domain $\Omega$ which satisfy\footnote{
Note that
$
(\nabla B-(\nabla B)^T)\cdot B = \begin{cases}  J B^\perp & d=2\\
  J \times B & d=3
  \end{cases}
$
 where  $J=\curl B$ if $d=3$ and
 $J = \nabla^\perp \cdot B$ if $d=2$.}
  \begin{alignat}{2}
   \nabla\cdot(\rho u\otimes  u) -  (\nabla B-(\nabla B)^T)\cdot B &= -\nabla q &&\qquad \text{ in } \Omega,\label{vireqn}\\
   q &=0 \qquad && \qquad \text{ at } \pa \Omega,\label{vireqn4}\\
    \rho u\cdot \hat{n} = B\cdot \hat{n} &= 0 \qquad &&\qquad \text{ at } \pa \Omega.\label{vireqn5}
 \end{alignat}
free boundary equilibria of incompressible MHD (when $\rho=1$, $q=p$) and compressible MHD when $\rho$ is the density and $q=P$ for a given equation of state $P=P(\rho)$, satisfy \eqref{vireqn}--\eqref{vireqn5}.
  We prove

 \begin{prop}
Suppose that
 $\Omega\subset \mathbb{R}^d$ is a compact domain with $C^1$ boundary. If $(\rho, u, B)$ satisfy \eqref{vireqn}--\eqref{vireqn5} then
   \begin{equation}
   \int_\Omega \left(d q +\left(\frac{d}{2} -1\right)|B|^2  + \rho |u|^2\right)\rmd x = 0.
   \label{ident2}
  \end{equation}
 In particular, if $d\geq 2$ and  $\int_\Omega q>0$ then there are no $(\rho, u, B)$ which satisfy \eqref{vireqn}--\eqref{vireqn5}.
\end{prop}

 \begin{proof}

Defining
$
  p = q +\frac{1}{2}|B|^2,
$
  \eqref{vireqn} can be written in divergence form,
 \begin{equation}
  \nabla\cdot  ( \rho u \otimes u - B\otimes B) +\nabla p =0.
  \label{divform}
 \end{equation}
 We claim that if $\Omega$ is compact with $C^1$ boundary, the for any $j =1,...,d$ we have
 \begin{equation}
  \int_\Omega\left( p - |B_j|^2 +  \rho |u_j|^2\right)\rmd x= 0.
  \label{ident0}
 \end{equation}
Summing \eqref{ident0} over
  $j = 1,..., d$ we obtain \eqref{ident2}.

  To prove \eqref{ident0}, for each $j=1,...,d$, we contract the $j$th component of \eqref{divform} with the vector
  $w^{(j)} = x^j e_j$ (note that $e_j$ is the unit vector along the $j$th coordinate axis).  We have
  \begin{equation}
  0 =  \pa_i(-\rho u^i u_j + B^i B_j - \delta^i_j p) w^{(j)}
  = \pa_i\left( -\rho u^i u_j w^{(j)} + B^i B_j w^{(j)}- \delta^i_j pw^{(j)}\right)
  + \rho |u_j|^2 - |B_j|^2 + p.
   \label{}
  \end{equation}
  Integrating the above over $\Omega$ and using the boundary conditions \eqref{vireqn4}--\eqref{vireqn5} gives \eqref{ident0}.
  \end{proof}

 \subsection*{Acknowledgments}   We thank  E. Rodriguez,  E. Paul,  and  W. Sengupta  for insightful remarks.
 We thank D. Peralta-Salas for pointing out the proof of Lemma \ref{lemcoils} using Cauchy–Kovalevskaya and informing us of the works \cite{KKPS20,V92} related to period functions for Hamiltonian systems.
 We  thank Y. Sire for helpful discussions
 about rigidity results for free boundary problems.  We also thank the referees for very useful comments on the manuscript.
 The work of PC was partially supported by NSF grant DMS-1713985 and by the
Simons Center for Hidden Symmetries and Fusion Energy  award \# 601960.
Research of TD was partially supported by  NSF grant DMS-2106233.
Research of DG was partially supported by
the Simons Center for Hidden Symmetries and Fusion Energy.

\end{document}